\documentclass[11pt,a4paper]{amsart}

\usepackage[T1]{fontenc}
\usepackage[utf8]{inputenc}

\usepackage[english]{babel}
\usepackage{amsmath,amsfonts,amssymb,amsthm}

\usepackage[pdftex]{graphicx}
\usepackage{lmodern}

\usepackage[foot]{amsaddr}

\usepackage{mathtools}
\mathtoolsset{showonlyrefs,showmanualtags}	

\DeclareMathAlphabet{\mathbbold}{U}{bbold}{m}{n}	

\theoremstyle{plain}
\newtheorem{theorem}{Theorem}
\newtheorem*{theorem*}{Theorem}
\newtheorem{prop}[theorem]{Proposition}

\newtheorem{lemma}[theorem]{Lemma}
\newtheorem*{lemma*}{Lemma}
\theoremstyle{definition}

\theoremstyle{remark}

\newtheorem{rmk}[theorem]{Remark}

\usepackage{hyperref}
\usepackage{xcolor}
\hypersetup{
    colorlinks,
    linkcolor={red!50!black},
    citecolor={blue!50!black},
    urlcolor={blue!80!black}
}

\usepackage[alphabetic, abbrev]{amsrefs}

\newcommand{\mis}{\mathsf{m}}	
\newcommand{\G}{\mathbb{G}}
\newcommand{\R}{\mathbb{R}}

\usepackage{url}

\author[Luca Rizzi]{Luca Rizzi}
\address{Univ. Grenoble Alpes, CNRS, Institut Fourier, F-38000 Grenoble, France}
\email{\href{mailto:luca.rizzi@univ-grenoble-alpes.fr}{luca.rizzi@univ-grenoble-alpes.fr}}

\allowdisplaybreaks

\title[A counterexample to gluing theorems for MCP spaces]{A counterexample to gluing theorems for MCP metric measure spaces}


\begin{document}

\begin{abstract}
Perelman's doubling theorem asserts that the metric space obtained by gluing along their boundaries two copies of an Alexandrov space with curvature $\geq \kappa$ is an Alexandrov space with the same dimension and satisfying the same curvature lower bound. We show that this result cannot be extended to metric measure spaces satisfying synthetic Ricci curvature bounds in the $\mathrm{MCP}$ sense. The counterexample is given by the Grushin half-plane, which satisfies the $\mathrm{MCP}(0,N)$ if and only if $N\geq 4$, while its double satisfies the $\mathrm{MCP}(0,N)$ if and only if $N\geq 5$.
\end{abstract}

\maketitle


\section{Statement of the results}

Perelman's doubling theorem \cite[5.2]{Perelman} states that an $n$-dimensional Alexandrov space with curvature $\geq \kappa$ can be doubled along its boundary yielding an Alexandrov space with the same curvature lower bound and dimension. This result has been extended by Petrunin \cite[Theorem 2.1]{Petrunin} to the gluing of Alexandrov spaces with isometric boundaries. More precisely, we have the following result, of which Perelman's doubling theorem constitutes a particular case.

\begin{theorem}[Petrunin's gluing theorem]\label{t:gluing}
Let $X$ and $Y$ be Alexandrov spaces of the same dimension and having curvature $\geq \kappa$. Suppose that their boundaries are isometric. Then the space $Z$ obtained by gluing $X$ and $Y$ along an isometry of their boundaries is an Alexandrov space of curvature $\geq \kappa$.
\end{theorem} 

Theorem~\ref{t:gluing} is false when gluing along general codimension $1$ subsets, e.g.\ parts of the boundary. Indeed, take two flat triangles, both with a side $L$ of given length. The glued space along $L$ is in general a quadrilateral, which is an Alexandrov space if and only if it is convex.

\medskip

It is interesting to understand whether these classical results in Alexandrov geometry hold true for more general metric measure spaces $(X,d,\mis)$ satisfying synthetic curvature bounds in the sense of Lott-Sturm-Villani. It is necessary to mention that a notion of boundary for m.m.s.\ is not yet available in full generality. Nevertheless, any meaningful notion of boundary should agree with the classical one when $X$ is a smooth manifold with topological boundary $\partial X$, $d$ is a complete metric that is continuous with respect to the topology of $X$, and $\mis$ is a smooth measure (i.e., defined by a smooth, positive tensor density). We call spaces $(X,d,\mis)$ satisfying these conditions \emph{smooth m.m.s.\ with boundary}.

The \emph{measure contraction property} is one of the weakest synthetic curvature bounds, and it was introduced independently in \cite{Ohta1} and \cite{Sturm2}. For proper, complete, essentially non-branching m.m.s.\ equipped with a locally finite non-negative Borel measure, the two definitions are equivalent, see \cite[Appendix A]{CM-optimal}. This condition, briefly $\mathrm{MCP}(K,N)$, depends on two parameters $K$ and $N$, playing the role of a Ricci curvature lower bound and a dimensional upper bound, respectively. We recall that any $n$-dimensional complete Alexandrov space with curvature $\geq \kappa$ satisfies the $\mathrm{MCP}((n-1)\kappa,n)$, and that for an $n$-dimensional complete Riemannian manifold the condition $\mathrm{Ric}_g \geq \kappa$ is equivalent to the $\mathrm{MCP}(\kappa,n)$.

If $K =0$, and $N \geq 1$, a m.m.s.\ $(X,d,\mis)$ satisfies the $\mathrm{MCP}(0,N)$ if for all $x \in X$ there exists a measurable map $\phi : X \to \mathrm{Geo}(X)$ such that, letting $\phi_t = e_t \circ \phi$, one has $\phi_0 = x$, $\phi_1 = \mathrm{id}_X$, and the following inequality holds true for every measurable set $A \subseteq X$ with $0<\mis(A) < +\infty$:
\begin{equation}
\mis(\phi_t(A)) \geq t^N \mis(A), \qquad \forall t \in [0,1].
\end{equation}
Here, $\mathrm{Geo}(X)$ is the space of geodesics of $(X,d)$, that is
\begin{equation}
\mathrm{Geo}(X) = \{\gamma \in C([0,1],X) \mid d(\gamma_t,\gamma_s) = |t-s|d(\gamma_0,\gamma_1)\},
\end{equation}
and for all $t \in [0,1]$ we denoted by $e_t : \mathrm{Geo}(X) \to X$ the evaluation map.

The purpose of this note is to prove that Perelman's doubling theorem and its generalizations cannot be extended to m.m.s.\ satisfying the $\mathrm{MCP}$.

\begin{theorem}\label{t:nopet}
For any $N \geq 5$ there exists a smooth m.m.s.\ with boundary $(X,d,\mis)$ that satisfies the $\mathrm{MCP}(0,N-1)$ and such that its double satisfies the $\mathrm{MCP}(0,N)$ but not the $\mathrm{MCP}(0,N-\varepsilon)$, for all $\varepsilon>0$.
\end{theorem}

To prove Theorem~\ref{t:nopet}, we exhibit a smooth m.m.s.\ with boundary, the Grushin half-plane, satisfying the required properties for $N=5$. The general case follows by taking the product with Euclidean spaces of the appropriate dimension, and using the additivity of the dimensional parameter of the $\mathrm{MCP}$ under metric products \cite[Proposition 3.3]{Ohta2}. We do not know whether the bound $N \geq 5$ is optimal.

\subsection*{Open problem}
It is interesting to understand whether gluing-type theorems hold in the more restrictive setting of essentially non-branching m.m.s.\ with boundary satisfying the $\mathrm{CD}(K,N)$ condition. The latter, similarly to the Alexandrov condition, is known to be a local property \cite{CM-CD}.

\subsection*{Acknowledgements}
The problem of extending gluing theorems to $\mathrm{MCP}$ spaces, in connection with the Grushin structure, was brought to my attention by N. Gigli. I wish to thank him and K.-T. Sturm for their comments.

\section{The Grushin plane and half-plane}

The Grushin plane is the basic example of almost-Riemannian structure, introduced in \cite{Gauss-Bonnet}. We give here a self-contained presentation, and we refer to the monograph \cite{nostrolibro} for a systematic presentation and proofs. Consider on $\R^2$ the smooth vector fields
\begin{equation}\label{eq:frame}
X_1 = \partial_x, \qquad X_2 = x \partial_y,
\end{equation}
which are orthonormal with respect to the singular Riemannian metric
\begin{equation}\label{eq:metric}
g= dx \otimes dx + \frac{1}{x^2} dy \otimes dy.
\end{equation}
We say that a Lipschitz curve $\gamma:[0,1] \to \R^2$ is \emph{admissible} if there exist $u \in L^\infty([0,1],\R^2)$ such that
\begin{equation}
\dot\gamma(t) = \sum_{i=1}^2 u_i(t) X_i(\gamma(t)), \qquad a.e.\ t \in [0,1].
\end{equation}
We define the length of an admissible curve as
\begin{equation}
\ell(\gamma) = \int_0^1 \sqrt{ u_1^2(t)+ u_2^2(t)}\, dt. 
\end{equation}
The minimization of the length functional yields a metric structure $d$ on $\R^2$, which is complete and continuous with respect to the Euclidean topology. An important fact, well known in optimal control theory, is that all length-minimizing curves on $\G = (\R^2,d)$ are described by a degenerate Hamiltonian flow on the cotangent bundle, with Hamiltonian $H : T^*\G \to \R$ given explicitly in coordinates $(x,y,u,v)$ on $T^*\G$ by\footnote{Hereafter we employ standard canonical coordinates on $T^*\R^2$, in such a way that a covector $\lambda = u dx + v dy \in T_{(x,y)}^* \R^2$ has coordinates $(x,y,u,v)$.}
\begin{equation}
H = \frac{1}{2}(u^2+x^2 v^2).
\end{equation}
For all $q \in \G$, the \emph{exponential map} $\exp_q : T_q^* \G \to \G$ is defined by projecting the Hamiltonian flow on $\G$.  A crucial fact is that the aforementioned system is integrable. The explicit expression of $\exp$ is not necessary at this point and it is postponed to Section~\ref{s:exponentialmap}. The essential properties concerning our analysis are resumed in Proposition~\ref{p:grushingeod}. In Appendix~\ref{s:appendix} we provide a more detailed overview of the geodesic structure of the Grushin plane.

\begin{prop}[Grushin cotangent injectivity domain]\label{p:grushingeod}
For all $q =(x,y)\in \G$ there exists a smooth map $\exp_q : T_q^* \G \to \G$ and an open set
\begin{equation}
D_q :=  \{(u,v) \in T_q^*\G \mid  H(x,y,u,v) \neq 0, \quad |v| < \pi\} \subset T_q^* \G,
\end{equation}
such that, for all $q \in \G$:
\begin{itemize}
\item the restriction $\exp_q : D_q \to \exp_q(D_q)$ is a smooth diffeomorphism;
\item the set $\exp_q(D_q)$ is dense in $\G$, and its complement has zero measure;
\item for all $p \in \exp_q(D_q)$ there exists a unique geodesic joining $q$ with $p$, given by $\gamma_t= \exp_q(t\lambda)$ for a unique $\lambda \in D_q$.
\end{itemize}
\end{prop}
\begin{rmk}
The topology of $D_q$ is different depending on whether the initial point $q$ is located on the $y$-axis or not. This is related with the degeneracy of $H$ and the occurrence of singular geodesics in almost-Riemannian geometry. In particular, when $q=(0,y)$, then any curve $\exp_q(t\lambda)$, with $H(\lambda)=0$, corresponds to the same, trivial (and singular) geodesic. For this reason we remove the corresponding set of covectors from $D_q$.
\end{rmk}
Another known fact is the following, which is a result of the explicit structure of geodesics and the cut-locus of $\G$ (see Appendix~\ref{s:appendix}). 
\begin{prop}
The closed half-planes $\R_\pm^2 = \{(x,y) \in \R^2 \mid \pm x \geq 0\}$ are convex subsets of the Grushin plane: for all pairs $q,p \in \R_\pm^2$ there exists a unique $\gamma \in \mathrm{Geo}(\G)$ joining them and whose support lies in $\R_\pm^2$.
\end{prop}
Hence, the restriction of the Grushin distance to the half-planes defines a complete, locally compact, length metric space structure $\G_\pm = (\R_\pm^2,d)$, where the length functional is given by the restriction of the Grushin length to admissible curves whose support is contained in $\R^2_\pm$. Furthermore, all geodesics of $\G_\pm$ are precisely the geodesics of $\G$ whose support is contained in the corresponding half-plane.

Therefore, Proposition~\ref{p:grushingeod} holds true for $\G_\pm$, replacing $D_q$ with the sets
\begin{equation}
D_q^\pm = D_q \cap \exp_q^{-1}(\R^2_\pm), \qquad \forall q \in \G_\pm.
\end{equation}

We equip $\G$ and $\G_\pm$ with the Lebesgue measure inherited as subsets of $\R^2$. In particular $\G = (\R^2,d,\mathcal{L})$ and $\G_\pm = (\R^2_\pm,d,\mathcal{L})$ are smooth m.m.s. with boundary which, for $\G_\pm$, consists in the $y$-axis and has zero measure.

\subsection{Metric measure double}

Given two metric spaces $(X,d_X)$, $(Y,d_Y)$, two subsets $X' \subset X$, $Y' \subset Y$, and an isometry $f : X' \to Y'$, their gluing is the metric space $(Z,d_Z)$, where $Z = (X \sqcup Y)/(x \sim f(x))$, and
\begin{equation}
d_Z(p,q) = \begin{cases}
d_X(p,q) & p,q \in X, \\
d_Y(p,q) & p,q \in Y, \\
\inf_{a \in X'} d_X(p,a) + d_Y(f(a),q) & p \in X,\, q \in Y, \\
\inf_{a \in X'} d_Y(p,f(a)) + d_X(a,q) & p \in Y,\, q \in X.
\end{cases}
\end{equation}
In this case, we identify $X$, $Y$ and $X'=Y'$ isometrically as subsets of $Z$. 

If $(X,d_X,\mis_X)$ and $(Y,d_Y,\mis_Y)$ are m.m.s., equipped with positive Borel measures $\mis_X$ and $\mis_Y$ such that $\mis_X(X') = \mis_Y(Y') = 0$, then we can define a positive Borel measure on $(Z,d_Z)$ by setting, for all Borel sets $B \subset Z$,
\begin{equation}\label{eq:measure}
\mis_Z(B): = \mis_X(B \cap X) + \mis_Y(B \cap Y).
\end{equation}
The double of a m.m.s.\ $(X,d,\mis)$ along a zero-measure subset $X' \subset X$ is defined as the gluing of two copies of $(X,d)$ via the isometry $\mathrm{id} : X' \to X'$, equipped with the measure defined in \eqref{eq:measure}.

\medskip

Consider now the Grushin plane $\mathrm{G}$ and half-planes $\G_\pm$. Clearly, the gluing of $\G_+$ and $\G_-$ along their boundary is isomorphic to the double of $\G_\pm$. This construction recovers the original Grushin structure.

\begin{prop}\label{p:gluingrush}
The gluing of the two Grushin half-planes $\G_\pm$, equipped with their Lebesgue measure, via the trivial isometry $\mathrm{id} : \partial \G_- \to \partial \G_+$ of their boundaries is the Grushin plane $\G$, equipped with the Lebesgue measure.
\end{prop}
\begin{proof}
This is immediate, since the metric and measure structures on $\G_\pm$ are obtained by restriction of the ones of $\G$. For a proof of the metric part, see e.g.\ \cite[Lemma 5.24(4)]{bookmetric}. 
\end{proof}

\subsection{Exponential map}\label{s:exponentialmap}
We give now the explicit formula for the exponential map of the Grushin structure, which will be needed for the proofs. Letting $q =(x,y)$ and $\lambda =(u,v) \in T^*_q \G$ we have $\exp_q(t \lambda) = (x_t,y_t)$ with
\begin{align}
x_t & = x \cos(t v) + u \frac{\sin(t v)}{v}, \\
y_t & = y + \frac{\sin (2 t v) \left(v^2 x^2-u^2\right)+2 v \left(t \left(v^2 x^2+u^2\right)+u x-u x \cos (2 t v)\right)}{4 v^2}.
\end{align}
The above formulas are well defined by their real-analytic continuation at $v=0$, that is $x_t = u t$ and $y_t = y$. With the same convention, the Jacobian determinant of the exponential map is
\begin{equation}\label{eq:jacdet}
J(x,y,u,v) = 
\frac{\left(u^2+ u v^2 x+v^2 x^2\right) \sin ( v) - u^2 v \cos (v)}{v^3}.
\end{equation}
We stress that, for all $q=(x,y) \in \G$, $(u,v) \in D_q$, and $t \in [0,1]$, we have $J(x,y,tu,tv) >0$, for all $t \in [0,1]$. This means that rays $[0,1]\ni t \mapsto (tu,tv) \in T_q \G$ are all regular points of the exponential map.

\section{Proof of the main result}

We refer to \cite{Ohta1} for the definition of $\mathrm{MCP}(K,N)$ for $K \neq 0$. Recall that $\mathrm{MCP}(K,N)$ implies $\mathrm{MCP}(K',N')$ for $K'< K$ and $N'>N$ (see \cite[Lemma 2.4]{Ohta1}).  Since complete $\mathrm{MCP}(K,N)$ spaces with $K>0$ are bounded, the Grushin plane or half-planes can only satisfy $\mathrm{MCP}(K,N)$ for $K \leq 0$. The following fact was proved in \cite[Theorem 10]{BR-inequalities}.
\begin{theorem}\label{t:1}
The Grushin plane $\G$, equipped with the Lebesgue measure, satisfies the $\mathrm{MCP}(K,N)$ if and only if $K \leq 0$ and $N \geq 5$.
\end{theorem}
As a matter of fact, the bound $N \geq 5$ in the proof of \cite{BR-inequalities} comes from contraction along geodesics crossing the singular region $\{x=0\}$, which are absent in the half-plane. Hence, one could hope that the latter satisfy a stronger $\mathrm{MCP}(0,N)$. This is the case, and we have the following.
\begin{theorem}\label{t:2}
The Grushin half-planes $\G_\pm$, equipped with the Lebesgue measure, satisfy the $\mathrm{MCP}(K,N)$ if and only if $K \leq 0$ and $N \geq 4$.
\end{theorem}
By Proposition~\ref{p:gluingrush} and Theorems~\ref{t:1}-\ref{t:2}, the Grushin half-plane is the counterexample to Perelman's doubling theorem claimed in Theorem~\ref{t:nopet}, with $N=5$. Hence, it only remains to prove Theorem~\ref{t:2}. 

\subsection{Proof of Theorem~\ref{t:2}}

\textbf{Step 1.\ Reduction to the case $K=0$.} We remarked that $\mathrm{MCP}(0,N)$ implies $\mathrm{MCP}(K,N)$ for all $K < 0$. On the contrary, assume that $\G$ satisfies the $\mathrm{MCP}(K,N)$ for some $K < 0$ and $N>1$. The scaled spaces $\G^\varepsilon = (\R^{2}, \varepsilon^{-1}d,\varepsilon^{-3}\mathcal{L})$ verify the $\mathrm{MCP}(\varepsilon^2 K,N)$ by \cite[Lemma 2.4]{Ohta1}. But $\G$ and $\G^\varepsilon$ are isomorphic through the dilation $\delta_\varepsilon(x,y) = (\varepsilon x, \varepsilon^2 y)$, with $\varepsilon>0$, hence $\G$ satisfies the $\mathrm{MCP}(\varepsilon^2 K, N)$ for all $\varepsilon>0$. Taking the limit $\varepsilon \to 0^+$ we obtain that $\G$ satisfies the $\mathrm{MCP}(0,N)$. We conclude that $\mathrm{MCP}(K,N) \Leftrightarrow \mathrm{MCP}(0,N)$ for $\G$ and $\G_\pm$. It is left to prove that, for $\G_\pm$, the $\mathrm{MCP}(0,N)$ is verified for $N =4$ and false for $N<4$.

\textbf{Step 2.\ Proof of the case $K=0$.} Let $q \in X$, where $X= \G_\pm$ (but nothing changes in the first part of the proof if $X=\G$, replacing $D_q^\pm$ with $D_q$). By Proposition~\ref{p:grushingeod}, up to modification on a negligible set, there exists a unique measurable map $\phi : X \to \mathrm{Geo}(X)$ such that, letting $\phi_t = e_t \circ \phi$, one has $\phi_0 = q$, $\phi_1 = \mathrm{id}_X$. Indeed, if $p = \exp_q(\lambda)$, then $\phi_t(p) = \exp_q(t\lambda)$. 


Since $\phi_t$ is, up to a negligible set, a smooth diffeomorphism, the proof of the $\mathrm{MCP}(0,N)$ is equivalent to an inequality for the Jacobian determinant of the Grushin (half-)plane. Up to a negligible set, any Borel set $A \subset X$ can be written as $A = \exp_q(\bar{A})$ for $\bar{A} \subset T_q^*X$. Then,
\begin{align}
\mis(\phi_t(A)) & = \int_{\exp_q(t\bar{A})} dx dy = \int_{\bar{A}} t^2 J(x,y,tu,tv)\, du dv \\
& =t^2 \int_{\bar{A}} \frac{J(x,y,tu,tv)}{J(x,y,u,v)} J(x,y,u,v)\, du dv \\
& \geq t^2 \inf_{u,v} \frac{J(x,y,tu,tv)}{J(x,y,u,v)}\, \mis(A), \qquad \forall t \in [0,1],
\end{align}
where the infimum is computed over all $(u,v) \in D_q$ if $X=\G$ (resp.\ $D_q^\pm$ if $X = \G_\pm$). Therefore, the proof (or the negation) of the $\mathrm{MCP}(0,N)$ is equivalent to the proof (or the negation) of the inequality
\begin{equation}\label{eq:toprove}
\frac{J(x,y,tu,tv)}{J(x,y,u,v)} \geq t^{N-2}, \qquad \forall t \in [0,1],
\end{equation}
for all $(x,y) \in \R^2$ (resp.\ $\R_\pm^2$) and $u,v \in D_q$ (resp.\ $D_q^\pm$). We restrict now to the Grushin half-planes, and Theorem~\ref{t:2} follows from the next lemma.
\begin{lemma}\label{l:ineq}
Inequality \eqref{eq:toprove} holds true for all $t \in[0,1]$, $(x,y) \in \R^2_\pm$ and $(u,v) \in D_q^\pm$ if $N =4$, and is false if $N<4$.
\end{lemma}
\begin{proof}
Without loss of generality we consider $\G_+$, hence for $q=(x,y) \in \G_+$ we have $x\geq 0$. The argument is similar to \cites{R-MCP,BR-Htype}, with the necessary modifications for the Grushin plane.

\smallskip

\textbf{Characterization of $D_q^+$.} Recall that $D_q^+ = D_q \cap \exp_q^{-1}(\R_+^2)$ is characterized by the property that geodesic $\exp_q(t\lambda)$ with $\lambda=(u,v) \in D_q^+$ remain in $\R^2_+$. Using the explicit expressions for geodesics in Section~\ref{s:exponentialmap}, we obtain
\begin{equation}\label{eq:conditionD}
(u,v) \in D_q^+ \quad \Leftrightarrow \quad \begin{cases} (u,v) \neq (0,0), \\
|v| < \pi, \\
x \cos(t v) + u \frac{\sin(t v)}{v} \geq 0, \qquad  \forall t \in[0,1].
\end{cases}
\end{equation}
Notice that the last condition is true if and only if it is true for $t=1$ (this is a consequence of the convexity of $\G_+$). We consider two cases.

\textbf{Case $v=0$.} This case corresponds to straight horizontal rays $(x_t,y_t) = (x+tu,y) \in \G_+$ and yields the necessity of the lower bound $N\geq 4$. In this case the trigonometric terms disappear in the Jacobian determinant \eqref{eq:jacdet}, and inequality \eqref{eq:toprove} is
\begin{equation}\label{eq:toprovev0}
\frac{J(x,y,tu,0)}{J(x,y,u,0)} = \frac{t^2 u^2 +3t u x +3x^2}{u^2 +3ux + 3x^2} \geq t^{N-2}, \qquad \forall t \in[0,1],
\end{equation}
to be proved for all $(u,0) \in D_q^+$, that is for all non-zero $u \geq -x$. Both sides are strictly positive for all $t\in [0,1]$, hence \eqref{eq:toprovev0} is equivalent to the following
\begin{equation}
\int_{tu}^u \frac{d}{dz} \log f_{x}(z) \, dz \leq (N-2) \int_{tu}^u \frac{d}{dz}\log |z| \, dz, \quad \forall t \in (0,1),\; u \geq -x,
\end{equation}
where $f_{x}(u) := u^2 + 3 u x + 3 x^2$. This inequality is equivalent to the corresponding inequality for the integrands, which is:
\begin{equation}\label{eq:critical}
(N-4) u^2 + 3x (N-3) u + 3x^2(N-2) \geq 0, \qquad \forall x>0,\, u \geq -x.
\end{equation}
The above is clearly violated if $N < 4$, and easily verified if $N=4$. This implies the ``only if'' part of the statement. To conclude, we prove the inequality \eqref{eq:toprove} in the remaining case $v \neq 0$, fixing $N=4$.

\begin{rmk} We observe here that, for the verification of the $\mathrm{MCP}(0,N)$ in the case of the \emph{full} Grushin plane, everything so far would be the same, except for the absence of the constraint $z \geq -x$. In this case \eqref{eq:critical} is violated for $N=4$ for geodesics $(x_t,y_t) = (x +ut,y)$ that cross the $y$-axis horizontally and travel far enough in the opposite side. In this case one can verify that \eqref{eq:critical} is still verified but only if $N \geq 5$.
\end{rmk}

\textbf{Case $v \neq 0$.} By symmetry, we actually assume $v >0$. Let first $u =0$, corresponding to geodesics $(x_t,y_t) = (x \cos(tv),y+x^2(\sin(2vt) + 2vt))$, never crossing the $y$-axis. In this case
\begin{equation}
\frac{J(x,y,0,tv)}{J(x,y,0,v)}= \frac{\sin(t v)}{t \sin(v)} \geq 1 \geq t^{N-2}, \qquad \forall t \in [0,1].
\end{equation}
Consider now $u \neq 0$. We introduce the new variable 
\begin{equation}
a:= \frac{vx}{u} \in \R_0 = \R \setminus \{0\}.
\end{equation}
In terms of this new variable
\begin{equation}
\frac{J(x,y,tu,tv)}{J(x,y,u,v)} = \frac{f_a(t v)}{t f_a(v)}, \quad \text{where} \quad f_a(v):= (1+a v + a^2)\sin(v) - v \cos(v).
\end{equation}
It remains to prove that for $N=4$, and for all allowed values of $a$ and $v$ (with an abuse of notation, we write $(a,v) \in D_q^+$) we have
\begin{equation}\label{eq:remains}
\frac{f_a(tv)}{f_a(v)} \geq t^{N-1}, \qquad \forall t \in [0,1].
\end{equation}
Recalling that $0<v<\pi$, both sides of \eqref{eq:remains} are strictly positive on $t \in (0,1]$, we can take the logarithms and the inequality is equivalent to
\begin{equation}
\int_{tv}^{v} \frac{d}{d z} \log f_{a}(z) \, dz \leq (N-1) \int_{tv}^{v} \frac{d}{dz}\log |z| \, dz, \qquad \forall t \in (0,1),\; a \in \R_0.
\end{equation}
The above inequality is equivalent to the corresponding one for the integrands. After some computation, we obtain the equivalent inequality
\begin{equation} \label{eq:troiaboia}
c_2(v) a^2 + v c_1(v) a +c_0(v) \geq 0, \qquad \forall (a,v) \in D_q^+,
\end{equation}
where we defined
\begin{gather}
c_2(v)  := 3 \sin (v)-v \cos (v), \qquad c_1(v) := 2  \sin (v)-v \cos (v), \\
c_0(v)  := 3 \sin (v)-3 v \cos (v) -v^2 \sin (v) = \int_0^v z (\sin(z)- z \cos(z)) \,dz.
\end{gather}
Thanks to the inequality $\sin(v) -v\cos(v) \geq 0$ holding for $v \in (0,\pi)$, we have
\begin{equation}
c_2(v) \geq 2 v \cos(v), \quad \text{and} \quad c_0(v) \geq 0.
\end{equation}
Hence \eqref{eq:troiaboia} is implied by the easier inequality
\begin{equation}\label{eq:easier}
2 a^2 \cos(v)  +  a (2 \sin(v) - v \cos(v)) \geq 0, \qquad \forall (a,v) \in D_q^+.
\end{equation}
The condition $(a,v) \in D_q^+$ which, more precisely, is the rewriting of \eqref{eq:conditionD} in terms of the variables $a$ and $v$, reads
\begin{equation}\label{eq:conditionina}
a^2 \cos(v) + a \sin(v) \geq 0, \qquad v \in (0,\pi), \qquad a \in \R_0.
\end{equation}
In the case $a >0$ (corresponding to geodesics moving initially to the right, possibly turning back at an intermediate time) inequality \eqref{eq:easier} is trivially verified, only using the condition $v \in (0,\pi)$. In the remaining case $a <0$ (corresponding to geodesics moving to the left, towards the $y$-axis of $\G_+$), the inequality \eqref{eq:easier} is also verified, invoking \eqref{eq:conditionina}. This concludes the proof.
\end{proof}

\section{A comment on Bakry-Emery curvature}\label{s:comment}

With the exception of the singular region $\{x=0\}$, the metric structure on the Grushin plane is Riemannian. In particular $\mathbb{G}_{\pm} = (\R^2_\pm,d,\mathcal{L})$ can be seen as a Riemannian smooth m.m.s.\ with weighted measure 
\begin{equation}
\mathcal{L} = e^{-V} \mathrm{vol}_g, \qquad V= -\log|x|.
\end{equation}
For weighted Riemannian m.m.s.\ a sufficient condition for the $\mathrm{MCP}(0,N)$ (and also for the stronger curvature-dimension condition $\mathrm{CD}(0,N)$), is the non-negativity of the Bakry-Emery Ricci curvature $\mathrm{Ric}_{N,V}$. See for example \cite[Theorem 14.8]{V-oldandnew}. For a general weighted $n$-dimensional Riemannian m.m.s.\ $(M,d_g, e^{-V} \mathrm{vol}_g)$ with $V \in C^\infty(M)$, the Bakry-Emery Ricci curvature is given by the formula
\begin{equation}
\mathrm{Ric}_{N,V} = \mathrm{Ric}_g + \mathrm{Hess}(V) - \frac{dV \otimes dV}{N-n}, \qquad N >n.
\end{equation}
For the specific case of the Grushin metric \eqref{eq:metric} we obtain, in terms of the orthonormal frame \eqref{eq:frame}, the following formulas for the Levi-Civita connection:
\begin{equation}
\nabla_{X_1} X_1 = \nabla_{X_1}X_2 = 0, \qquad \nabla_{X_2} X_1 = -\frac{1}{x} X_2, \qquad \nabla_{X_2}X_2 = \frac{1}{x} X_1,
\end{equation}
from which we can compute the following quantities:
\begin{equation}
\mathrm{Ric}_g = -\frac{2}{x^2} g, \qquad \mathrm{Hess}(V) = \frac{1}{x^2} g, \qquad dV \otimes dV = \frac{1}{x^2} dx \otimes dx.
\end{equation}
It follows that the Bakry-Emery Ricci curvature for $\G_\pm$, is
\begin{equation}
\mathrm{Ric}_{N,V} = -\frac{1}{x^2} g - \frac{1}{x^2} \frac{dx \otimes dx}{N-n}, \qquad N > n.
\end{equation}
The above is never non-negative, actually $\mathrm{Ric}_{N,V} \leq - \tfrac{1}{x^2} g$. Therefore, we can conclude that Theorems~\ref{t:1}-\ref{t:2} do not follow from Bakry-Emery comparison.

\appendix

\section{Grushin geodesics and cut-locus}\label{s:appendix}

Following \cite[Section 3.2]{Gauss-Bonnet}, we give a more detailed presentation of the geodesic structure of the Grushin plane.

For all $q \in \G$ and $(u,v) \in T_q^*\G$, we consider \emph{rays} $t \mapsto \exp_q(tu,tv)$, for all $t \geq 0$. Rays are length-minimizing up to the first time $t_*>0$ at which $(tu,tv)$ meets the boundary of $\bar{D}_q$ (that is, the lines $v=\pm \pi$), which is equal to $t_* = \pi/|v|$, for $v \neq 0$, and $t_* = +\infty$ otherwise. See Proposition~\ref{p:grushingeod}. 

The set $\mathrm{Cut}(q) \subset \G$ where rays from $q$ cease to be length-minimizing is called \emph{cut-locus} of $q$. We refer to the explicit expressions of the exponential map in Section~\ref{s:exponentialmap}. Since reflections w.r.t.\ the $y$-axis and vertical translations are isometries of $\G$, it is sufficient to analyse two qualitatively different cases: $q=(0,0)$, $q=(-1,0)$, and the special case of straight lines.

\subsection{Straight lines} 

Rays corresponding to $v=0$ are straight horizontal lines $(x_t,y_t) = (x + tu, y)$, and are length-minimizing for all times and initial points. In general, rays corresponding to $v<0$ are reflection w.r.t.\ the $x$-axis of rays corresponding to $v>0$. Moreover, rays are determined by $(u,v)$ up to rescaling. Hence for the remaining cases we set $v=\pm 1$.

\subsection{Case $q=(0,0)$}

The corresponding rays are $(x_t,y_t)$, with
\begin{equation}
x_t = u \sin(t), \qquad y_t = \mathrm{sgn}(v)\tfrac{u^2}{4} \left(2t -\sin (2 t )\right).
\end{equation}
If $u=0$, these rays corresponds to the trivial geodesic. Otherwise, the two rays corresponding to $\pm u$ meet at a point $p$ on the $y$-axis, at time $t_* = \pi$, where they cease to be length-minimizing. There are exactly two geodesics between $q = (0,0)$ and $p$. In this case the closure of $\mathrm{Cut}(q)$ coincides with the $y$-axis, including hence $q$ itself, a phenomenon that never occurs in Riemannian geometry. See Figure~\ref{f:fig}.

\subsection{Case $q=(1,0)$}

The corresponding rays are $(x_t,y_t)$, with
\begin{align}
x_t  = \cos(t ) - u \sin(t), \qquad y_t  = \mathrm{sgn}(v) \tfrac{\sin (2 t ) \left(1 -u^2\right)+2  \left(t \left(1 +u^2\right)-u +u  \cos (2 t )\right)}{4}.
\end{align}
Let $u \neq 0$. The two rays corresponding to $\pm u$ meet at a point $p$, located to the left side of the $y$-axis, belonging to the set 
\begin{equation}
\mathrm{Cut}(q) = \{(-x_c,y_c + s \mid s \geq 0\} \cup \{(-x_c,-y_c-s \mid s \geq 0\},
\end{equation}
with $x_c=1,y_c=\pi/2$. There are exactly two geodesics between $q$ and $p$. In the special case $u=0$, the corresponding ray meets $\mathrm{Cut}(q)$ at its boundary (at $(-x_c,y_c)$ if $v>0$ and at $(-x_c,-y_c)$ if $v<0$), and there is only one geodesic joining $q$ with each boundary point of $\mathrm{Cut}(p)$. See Figure~\ref{f:fig}.

\begin{figure}
\includegraphics[width=.4\textwidth]{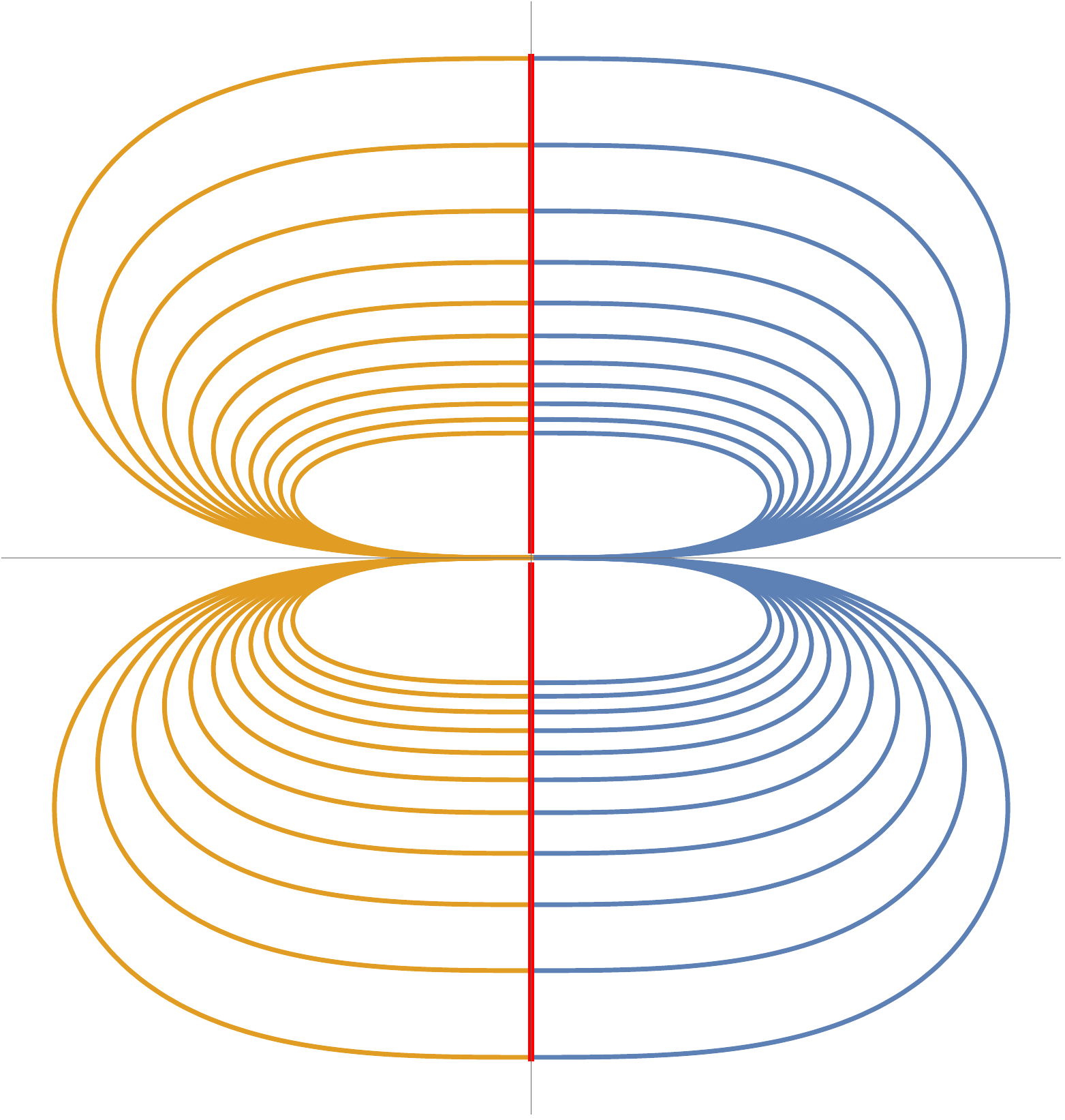} \qquad\qquad \includegraphics[width=.4\textwidth]{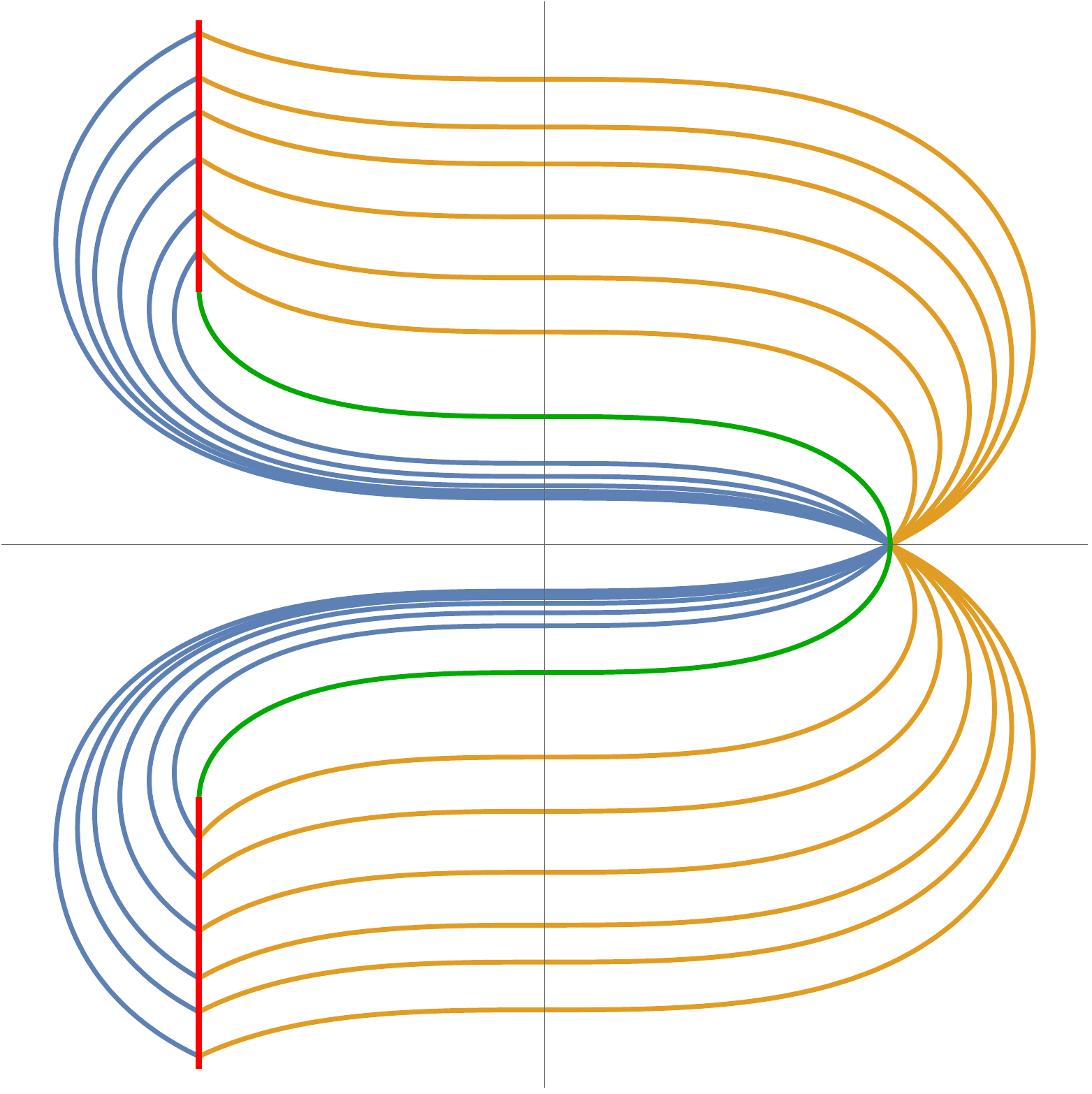}
\caption{Geodesics and cut-locus (in red) of the Grushin plane starting from the origin and from $q=(1,0)$.} \label{f:fig}
\end{figure}

\section*{Acknowledgements}
This work was supported by the Grant ANR-15-CE40-0018 of the ANR, by the ANR project ANR-15-IDEX-02, and by a public grant as part of the Investissement d'avenir project, reference ANR-11-LABX-0056-LMH, LabEx LMH, in a joint call with the ``FMJH Program Gaspard Monge in optimization and operation research''.

\bibliographystyle{alpha}
\bibliography{gluing}

\begin{bibdiv}
\begin{biblist}

\bib{nostrolibro}{article}{
      author={Agrachev, Andrei},
      author={Barilari, Davide},
      author={Boscain, Ugo},
       title={Introduction to {R}iemannian and sub-{R}iemannian geometry
  ({L}ecture {N}otes).
  \url{http://webusers.imj-prg.fr/~davide.barilari/Notes.php}},
        date={2017},
         url={http://webusers.imj-prg.fr/~davide.barilari/Notes.php},
}

\bib{Gauss-Bonnet}{article}{
      author={Agrachev, Andrei},
      author={Boscain, Ugo},
      author={Sigalotti, Mario},
       title={A {G}auss-{B}onnet-like formula on two-dimensional
  almost-{R}iemannian manifolds},
        date={2008},
        ISSN={1078-0947},
     journal={Discrete Contin. Dyn. Syst.},
      volume={20},
      number={4},
       pages={801\ndash 822},
         url={https://doi.org/10.3934/dcds.2008.20.801},
      review={\MR{2379474}},
}

\bib{bookmetric}{book}{
      author={Bridson, Martin~R.},
      author={Haefliger, Andr\'e},
       title={Metric spaces of non-positive curvature},
      series={Grundlehren der Mathematischen Wissenschaften [Fundamental
  Principles of Mathematical Sciences]},
   publisher={Springer-Verlag, Berlin},
        date={1999},
      volume={319},
        ISBN={3-540-64324-9},
         url={https://doi.org/10.1007/978-3-662-12494-9},
      review={\MR{1744486}},
}

\bib{BR-Htype}{article}{
      author={{Barilari}, D.},
      author={{Rizzi}, L.},
       title={{Sharp measure contraction property for generalized H-type Carnot
  groups}},
        date={2017-02},
     journal={Commun. Contemp. Math.},
         url={https://doi.org/10.1142/S021919971750081X},
        note={in press},
}

\bib{BR-inequalities}{article}{
      author={{Barilari}, D.},
      author={{Rizzi}, L.},
       title={{Sub-Riemannian interpolation inequalities: ideal structures}},
        date={2017-05},
     journal={ArXiv e-prints},
      eprint={1705.05380},
}

\bib{CM-CD}{article}{
      author={{Cavalletti}, F.},
      author={{Milman}, E.},
       title={{The Globalization Theorem for the Curvature Dimension
  Condition}},
        date={2016-12},
     journal={ArXiv e-prints},
      eprint={1612.07623},
}

\bib{CM-optimal}{article}{
      author={Cavalletti, Fabio},
      author={Mondino, Andrea},
       title={Optimal maps in essentially non-branching spaces},
        date={2017},
        ISSN={0219-1997},
     journal={Commun. Contemp. Math.},
      volume={19},
      number={6},
       pages={1750007, 27},
         url={https://doi.org/10.1142/S0219199717500079},
      review={\MR{3691502}},
}

\bib{Ohta1}{article}{
      author={Ohta, Shin-ichi},
       title={On the measure contraction property of metric measure spaces},
        date={2007},
        ISSN={0010-2571},
     journal={Comment. Math. Helv.},
      volume={82},
      number={4},
       pages={805\ndash 828},
         url={http://dx.doi.org/10.4171/CMH/110},
      review={\MR{2341840}},
}

\bib{Ohta2}{article}{
      author={Ohta, Shin-Ichi},
       title={Products, cones, and suspensions of spaces with the measure
  contraction property},
        date={2007},
        ISSN={0024-6107},
     journal={J. Lond. Math. Soc. (2)},
      volume={76},
      number={1},
       pages={225\ndash 236},
         url={http://dx.doi.org/10.1112/jlms/jdm057},
      review={\MR{2351619}},
}

\bib{Perelman}{unpublished}{
      author={Perelman, G.},
       title={Alexandrov's spaces with curvatures bounded from below {II}},
        date={1991},
        note={Unpublished, a copy the manuscript can be found on Anton
  Petrunin's website},
}

\bib{Petrunin}{incollection}{
      author={Petrunin, Anton},
       title={Applications of quasigeodesics and gradient curves},
        date={1997},
   booktitle={Comparison geometry ({B}erkeley, {CA}, 1993--94)},
      series={Math. Sci. Res. Inst. Publ.},
      volume={30},
   publisher={Cambridge Univ. Press, Cambridge},
       pages={203\ndash 219},
         url={http://dx.doi.org/10.2977/prims/1195166129},
      review={\MR{1452875}},
}

\bib{R-MCP}{article}{
      author={Rizzi, Luca},
       title={Measure contraction properties of {C}arnot groups},
        date={2016},
        ISSN={0944-2669},
     journal={Calc. Var. Partial Differential Equations},
      volume={55},
      number={3},
       pages={Art. 60, 20},
         url={https://doi.org/10.1007/s00526-016-1002-y},
      review={\MR{3502622}},
}

\bib{Sturm2}{article}{
      author={Sturm, Karl-Theodor},
       title={On the geometry of metric measure spaces. {II}},
        date={2006},
        ISSN={0001-5962},
     journal={Acta Math.},
      volume={196},
      number={1},
       pages={133\ndash 177},
         url={http://dx.doi.org/10.1007/s11511-006-0003-7},
      review={\MR{2237207}},
}

\bib{V-oldandnew}{book}{
      author={Villani, C{\'e}dric},
       title={Optimal transport},
      series={Grundlehren der Mathematischen Wissenschaften [Fundamental
  Principles of Mathematical Sciences]},
   publisher={Springer-Verlag, Berlin},
        date={2009},
      volume={338},
        ISBN={978-3-540-71049-3},
         url={http://dx.doi.org/10.1007/978-3-540-71050-9},
        note={Old and new},
}

\end{biblist}
\end{bibdiv}

\end{document}